\documentclass[12pt]{article}
\usepackage{amssymb}
\usepackage{array}
\usepackage{latexsym}
\usepackage{enumerate}
\usepackage{amsmath}
\usepackage{amsfonts}
\usepackage{amsthm}
\usepackage{graphicx}
\usepackage{comment}
\usepackage{color}
\usepackage[english]{babel}
\usepackage[T1]{fontenc}
\newcommand{\cH}{\mathcal H}

\newcommand{\cS}{\mathcal S}
\usepackage{geometry}
\usepackage{hyperref}
\ifx\smallsetminus\undefined
\def\smallsetminus{\setminus}
\fi
\title{A characterization of  Hermitian varieties as codewords}
\author{A.~Aguglia\footnote{Dipartimento di Meccanica, Matematica e Management, Politecnico di Bari,
  Via Orabona 4, I-70126 Bari}
,\, D. Bartoli\footnote{Dipartimento di Matematica e Informatica, Universit\`a degli Studi di Perugia, Via Vanvitelli 1, 06123 Perugia, Italy}, \,L.~Storme,\footnote{Ghent University, Department of Mathematics, Krijgslaan 281, 9000 Ghent, Belgium},\, Zs.~Weiner\footnote{MTA--ELTE Geometric and Algebraic Combinatorics Research Group,
ELTE E\"otv\"os Lor\'and University, Budapest, Hungary, and
Department of Computer Science,
1117 Budapest, P\'azm\'any P.\ stny.\ 1/C, Hungary}}
\date{}
\theoremstyle{plain}
\newtheorem{prop}{Proposition}[section]
\newtheorem{theorem}[prop]{Theorem}
\newtheorem{definition}[prop]{Definition}
\newtheorem{proposition}[prop]{Proposition}
\newtheorem{corollary}[prop]{Corollary}
\newtheorem{result}[prop]{Result}
\newtheorem{lemma}[prop]{Lemma}
\newtheorem{property}[prop]{Property}
\theoremstyle{definition}
\newtheorem{remark}[prop]{Remark}

\def\cK{\mathcal K}

\def\cS{\mathcal S}

\def\cH{\mathcal H}

\newcommand{\PG}{\mathrm{PG}}

\newcommand{\cV}{\mathcal V}
\begin{document}

\maketitle

\begin{abstract}
It is known that the Hermitian varieties are codewords in the code defined by the points and hyperplanes of the projective spaces $\PG(r,q^2)$. In finite geometry, also quasi-Hermitian varieties are defined. These are sets of points of $\PG(r,q^2)$ of the same size as a non-singular Hermitian variety of $\PG(r,q^2)$,  having the same intersection sizes with the hyperplanes of $\PG(r,q^2)$.
In the planar case, this reduces to the definition of a unital. A famous result of Blokhuis, Brouwer, and Wilbrink states that every unital in the code of the points and lines of $\PG(2,q^2)$ is a Hermitian curve. We prove a similar result for the quasi-Hermitian varieties in $\PG(3,q^2)$, $q=p^{h}$,  as well as in  $\PG(r,q^2)$, $q=p$ prime, or $q=p^2$, $p$ prime, and $r\geq 4$.
\end{abstract}
{\bf Keywords}: Hermitian variety; incidence vector; codes of projective spaces; quasi-Hermitian variety.

{\bf MSC: 51E20, 94B05}

\section{Introduction}
Consider the non-singular Hermitian varieties $\cH(r,q^2)$ in $\PG(r,q^2)$. A non-singular Hermitian variety $\cH(r,q^2)$ in $\PG(r,q^2)$ is the set of absolute points of a Hermitian polarity of $\PG(r,q^2)$. Many properties of a non-singular Hermitian variety $\cH(r,q^2)$ in $\PG(r,q^2)$ are known. In particular, its size is $(q^{r+1}+(-1)^r)(q^{r}-(-1)^r)/(q^2-1)$, and its intersection numbers with the hyperplanes of $\PG(r,q^2)$ are equal to $(q^{r}+(-1)^{r-1})(q^{r-1}-(-1)^{r-1})/(q^2-1)$, in case the hyperplane is a non-tangent hyperplane to $\cH(r,q^2)$, and equal to $1+q^2(q^{r-1}+(-1)^r)(q^{r-2}-(-1)^r)/(q^2-1)$ in case the hyperplane is a tangent hyperplane to $\cH(r,q^2)$; see~\cite{S}.

Quasi-Hermitian varieties $\cV$ in $\PG(r,q^2)$ are  generalizations of the non-singular Hermitian variety $\cH(r,q^2)$ so that $\cV$ and $\cH(r,q^2)$ have the same size and the same intersection numbers with hyperplanes.

Obviously, a Hermitian variety $\cH(r,q^2)$   can be viewed as a trivial quasi--Hermitian variety and we call $\cH(r,q^2)$ the \emph{classical
quasi--Hermitian variety of $\PG(r,q^2)$}.
In the 2-dimensional case, $\cV$ is also known as the classical example of a {\em unital} of the projective plane $\PG(2,q^2)$.

As far as we know, the only known non-classical quasi-Hermitian
varieties of $\PG(r,q^2)$ were constructed in \cite{AA,ACK,B,DS,M,FP}.

In \cite{BBW}, it is shown that a unital in $\PG(2,q^2)$ is  a Hermitian curve if and only if it is in the $\mathbb{F}_p$-code spanned by the lines of $\PG(2,q^2)$, with $q=p^h$, $p$ prime and $h\in \mathbb{N}$.

In this article, we prove the following result.
\begin{theorem}\label{mainth}
A quasi-Hermitian variety $\cV$ of $\PG(r,q^2)$, with $r=3$ and $q=p^h>4$, or $ r\geq 4$, $q=p>4$, or $ r\geq 4$, $q=p^2$, $p>3$ prime, is classical if and only if it is in the $\mathbb{F}_p$-code spanned by the hyperplanes of $\PG(r,q^2)$.
\end{theorem}
Furthermore we consider  {\em singular quasi-Hermitian varieties}, that is point sets having the same number of points as a singular Hermitian variety $\cS$
 and for which each intersection number with respect to hyperplanes is
also an intersection number of $\cS$ with respect to hyperplanes.
We  show that Theorem \ref{mainth} also holds in the case in which $\cV$ is assumed to be a singular quasi-Hermitian variety of $\PG(r,q^2)$.

\section{Preliminaries}
A subset $\cK$ of $\PG(r,q^2)$ is a $k_{n,r,q^2}$ if $n$ is a fixed integer, with $1\leq n\leq q^2$, such that:
\begin{itemize}
\item [(i)] $|\cK|=k$;
\item [(ii)] $|\ell \cap \cK|=1$, $n$, or $q^2+1$ for each line $\ell$;
\item [(iii)]  $|\ell \cap \cK|=n$ for some line $\ell$.
\end{itemize}

A point $P$ of $\cK$ is {\em singular} if every line through $P$ is either a unisecant or a line of $\cK$. The set $\cK$ is called {\em singular} or {\em non-singular} according as it has singular points or not.

Furthermore, a subset $\cK$ of $\PG(r,q^2)$ is called {\em regular} if
\begin{itemize}
\item [(a)] $\cK$ is a $k_{n,r,q^2}$;
\item [(b)] $3\leq n \leq q^2-1$;
\item [(c)] no planar section of $\cK$ is the complement of a set of type $(0, q^2+1-n)$.
\end{itemize}

\begin{theorem}\cite[Theorem 19.5.13]{JH} \label{main1}
Let $\cK$ be a $k_{n,3,q^2}$ in $\PG(3,q^2)$,  where $q$ is any prime power and $n \neq \frac{1}{2}q^2+1$.  Suppose furthermore that every point in $\cK$ lies on at least one $n$-secant. Then  $n=q+1$ and $\cK$ is a non-singular Hermitian surface.
\end{theorem}

\begin{theorem} \cite[Theorem 23.5.19]{HT} \label{main2}
If $\cK$ is a regular, non-singular $k_{n,r,q^2}$, with $r\geq 4$ and $q>2$, then $\cK$ is a non-singular Hermitian variety.
\end{theorem}

\begin{theorem}{ \cite[ Th. 23.5.1]{HT}}\label{th:singularcase1}
If $\mathcal{K}$ is a singular $k_{n,3,q^2}$ in $\PG(3,q^2)$ with $3\leq n\leq q^2-1$, then  the following holds:
$\mathcal{K}$ is $n$ planes through a line or a cone with vertex a point and base $\mathcal{K}^{\prime}$
a plane section  of type
\begin{itemize}
\item[\rm I.] a unital;
\item[\rm II.] a subplane $\PG(2,q)$;
\item[\rm III.] a set of type $(0,n-1)$ plus an external line;
\item[\rm IV.] the complement of a set of type $(0,q^2+1-n)$.
\end{itemize}
\end{theorem}

\begin{theorem}{\cite[Lemma 23.5.2 and Th. 25.5.3]{HT}}\label{th:singularcase}
If $\mathcal{K}$ is a singular $k_{n,r,q}$ with $r\geq4$,  then the singular points of $\mathcal{K}$ form a subspace $\Pi_d$ of dimension $d$ and one of the following possibilities holds:
\begin{enumerate}
\item $d=r-1$ and $\mathcal{K}$ is a hyperplane;
\item $d=r-2$ and $\mathcal{K}$ consists of  $n>1$ hyperplanes through $\Pi_d$;
\item $d\leq r-3$ and $\mathcal{K}$ is equal to a cone $\Pi_d \mathcal{K}^{\prime}$, with $\pi_d$ as vertex and with $\mathcal{K}$ as base, where $\mathcal{K}^{\prime}$ is a non singular $k_{n,r-d-1,q}$.
\end{enumerate}
\end{theorem}

\begin{result}[\cite{SzW}]
\label{bmost}
Let ${\cal M}$ be a multiset in ${\rm \PG}(2,q)$, $17<q$,
so that the number of lines intersecting it in not $k\pmod{p}$
points is $\delta$, where  $\delta <  {3 \over 16} (q+1)^2$.
Then the number of non $k\pmod{p}$ secants through any point is at most
${\delta\over q+1} + {2\delta ^2\over (q+1)^3}$ or at least $q + 1-({\delta\over q+1} + {2\delta ^2\over (q+1)^3})$.
\end{result}

\begin{property} [\cite{SzW}]
\label{most}
Let $\cal M$ be a multiset in ${\rm
\PG}(2,q)$, $q=p^h$, where $p$ is prime. Assume that there are $\delta$
lines that intersect $\cal M$ in not $k\pmod{p}$ points.
If through a point there are more than $q / 2$ lines intersecting
 $\cal M$ in not $k\pmod{p}$ points, then there exists a value $r$ such that
 the intersection multiplicity of at least $2{\delta\over q + 1} + 5$ of these lines with $\cal M$
 is $r$.
 \end{property}

 \medskip

\begin{result} [\cite{SzW}]
\label{general}
Let $\cal M$ be a multiset in ${\rm
\PG}(2,q)$, $17< q$, $q=p^h$, where $p$ is prime.
Assume that the number of lines intersecting $\cal M$ in not
$k\pmod{p}$ points is $\delta$, where $\delta <
(\lfloor \sqrt q \rfloor +1)(q+1-\lfloor \sqrt q \rfloor )$.
Assume furthermore that Property \ref{most} holds.
Then there
exists a multiset $\cal{M}'$ with the property that it intersects every
line in $k\pmod{p}$ points and the number of different points in
 $({\cal M}\cup {\cal M}')\setminus ({\cal M}\cap {\cal M}')$ is
exactly $\lceil {\delta\over q+1}\rceil$.
\end{result}

\medskip
\begin{result}[\cite{SzW}]
\label{codewords}
Let $B$ be a proper point set in ${\rm \PG}(2, q)$, $17< q$. Suppose that $B$ is a codeword of the lines of ${\rm \PG}(2, q)$. Assume also that $|B| < (\lfloor \sqrt q \rfloor +1)(q+1-\lfloor \sqrt q \rfloor )$.
Then $B$ is a linear combination of at most $\lceil {|B| \over q+1} \rceil$ lines. \qed
\end{result}

\section{Proof of Theorem \ref{mainth}}\label{qHv}
Let $V$ be the vector space of dimension $q^{2r}+q^{2(r-1)}+\cdots+q^2+1$ over the prime field $\mathbb{F}_p$,  where the coordinate positions for the vectors in $V$ correspond to the points of $\PG(r,q^2)$ in some fixed order.
 If $S$ is a subset of points in $\PG(r,q^2)$ then let $v^{S}$ denote the vector in $V$ with coordinate 1 in the positions corresponding to the points in $S$ and with coordinate $0$ in all other positions; that is $v^S$ is the {\em characteristic vector} of $S$.
Let $\mathcal C_p$ denote the subspace of $V$ spanned by the characteristic vectors of all the hyperplanes in $\PG(r,q^2)$. This code $C_p$ is called   {\em the linear code of $\PG(r,p^2)$}.

From
   \cite[Theorem 1]{JK}, we know that the  characteristic vector $v^{\cV}$ of a Hermitian variety $\cV \in \PG(r,q^2)$ is in $ \mathcal C_p$.
So from now on, we will assume that  $\cV$ is a quasi-Hermitian variety in $\PG(r,q^2)$ and
$v^{\cV} \in {\mathcal C}_p$. In the remainder of this section we will show that $\cV$ is a classical Hermitian variety.

\bigskip
The next lemmas hold for $r\geq 3$ and for  any $q = p^h$, $p$ prime, $h\geq 1$.

\begin{lemma}\label{Lemma:rette}
Every line of $\PG(r,q^2)$, $q = p^h$, $p$ prime, $h\geq 1$,  meets $\cV$ in $1 \pmod{p}$ points.
\end{lemma}

\begin{proof}
We may express \[v^{\cV}=v^{H_1}+\cdots+v^{H_t},\] where $H_1,\ldots, H_t$ are (not necessarily distinct) hyperplanes of $\PG(r,q^2)$.
Denote by  $\cdot$ the usual dot product. We get $v^{\cV}\cdot v^{\cV}=|\cV|\equiv 1$ (mod $p$).
On the other hand
\[v^{\cV}\cdot v^{\cV}=v^{\cV}\cdot (v^{H_1}+\cdots+v^{H_t})\equiv t  \pmod{p},\]
since every hyperplane of $\PG(r,q^2)$ meets $\cV$ in $1\pmod{p}$ points.
Hence, we have $t\equiv 1$ (mod $p$).
Finally, for a line $\ell$ of $\PG(r,q^2)$,
\[v^{\ell} \cdot v^{\cV}= v^{\ell}\cdot (v^{H_1}+\cdots+v^{H_t}) \equiv t  \pmod{p},\]
as every line of $\PG(r,q^2)$ meets a hyperplane in $1$ or $q^2+1$ points.
That is, $|\ell \cap \cV|\equiv 1$ (mod $p$) and in particular no lines of $\PG(r,q^2)$ are external to $\cV$.
\end{proof}

\begin{remark}
The preceding proof also shows that $\cV$ is a linear combination of $1\pmod{p}$ (not necessarily distinct) hyperplanes, all having coefficient one.
\end{remark}

\begin{lemma}\label{CDS}
For every hyperplane $H$ of $\PG(r,q^2)$, $q = p^h$, $p$ prime, $h\geq 1$, the intersection $H \cap \cV$ is in the code of points and hyperplanes of $H$ itself.
\end{lemma}
\begin{proof}
Let $\Sigma$ denote the set of all hyperplanes of $\PG(r,q^2)$.
By assumption, \begin{equation}\label{car}
v^{\cV}=\sum_{H_i \in \Sigma} \lambda_i v^{H_i}.
\end{equation}
For every $H\in \Sigma$, let
$\pi$ denote a hyperplane of $H$; then $\pi=H_{j_1} \cap \cdots \cap H_{j_{q^2+1}}$, where $H_{j_1},\ldots,H_{j_{q^2+1}}$  are the hyperplanes of $\PG(r,q^2)$ through $\pi$.  We assume $H=H_{j_{q^2+1}}$.
For every hyperplane $\pi$ of $H$, we set \[\lambda_{\pi}= \sum_{k=1, \ldots, q^2+1} \lambda_{j_k},\] where  $\lambda_{j_k}$ is the coefficient in \eqref{car} of $v^{H_{j_k}}$ and $H_{j_k}$  is one of the $q^2+1$ hyperplanes through $\pi$.

Now, consider \begin{equation}\label{eq6} T=\sum_{\pi\in \Sigma' } \lambda_{\pi}v^{\pi},
\end{equation} where $\Sigma'$ is the set of all hyperplanes in $H$.
We are going to show that \[T=v^{\cV\cap H}.\]

In fact, it is clear that at the positions belonging to the points outside of $H$ we see zeros. At a position belonging to a point in $H$, we see
the original coefficients of $v^{\cV}$ plus $(|\Sigma'|-1) \lambda_{j_{q^2+1}}$. Note that this last term is $0\pmod{p}$, hence
$T=v^{\cV\cap H}$.
\end{proof}

\begin{corollary}
\label{subgeocodes}
For every subspace $S$ of $\PG(r,q^2)$, $q = p^h$, $p$ prime, $h\geq 1$, the intersection $S \cap \cV$ is in the code of points and hyperplanes of $S$ itself. \qed
\end{corollary}

\begin{remark}
\label{1modpremark}
Lemma \ref{CDS}  and Corollary \ref{subgeocodes} are valid for  $\cV$  any  set of  points in $\PG(r,q^2)$ whose incidence vector belongs to  the code of points and hyperplanes of $\PG(r,q^2)$.
In particular, it follows that for every plane $\pi$ the intersection $\pi \cap \cV$ is a codeword of the points and lines of $\pi$,  $\pi \cap \cV$ has size $1$ mod $\pmod{p}$ and so it is a linear combination
of $1$ mod $\pmod{p}$ not necessarily distinct lines.
\end{remark}
\begin{lemma} \label{lemma1}
Let $\ell$ be a line of $\PG(r,q^2)$. Then there exists at least one plane through $\ell$ meeting $\cV$ in  $\delta$ points, with $\delta \leq q^3 + q^2 +q +1$.
\end{lemma}
\begin{proof}
By way of contradiction, assume that all planes through $\ell$ meet $\cV$ in more than  $q^3 + q^2+q+1$ points. Set $x=|\ell \cap \cV|$.
We get
\begin{equation}\label{eq7}
\frac{(q^{r+1}+(-1)^{r})(q^{r}-(-1)^{r})}{q^2-1}> m (q^3+q^2+q+1-x)+x,
\end{equation}
where $m=q^{2(r-2)}+q^{2(r-3)}+ \cdots +q^2+1$ is the number of planes in $\PG(r,q^2)$ through $\ell$.
From \eqref{eq7}, we obtain $x>q^2+1$, a contradiction.
\end{proof}

\begin{lemma}
\label{maxindex}
 For each line $\ell$ of $\PG(r,q^2)$, $q>4$ and $q = p^h$, $p$ prime, $h\geq 1$,  either  $|\ell \cap \cV| \leq q+1$ or  $|\ell \cap \cV| \geq q^2-q+1$.
\end{lemma}
\begin{proof}
 Let $\ell$ be a line of $\PG(r,q^2)$ and let
 $\pi$ be a plane through $\ell$ such that $|\pi \cap \cV|\leq q^3+q^2+q+1$; Lemma \ref{lemma1} shows that such a plane exists. Set $B=\pi \cap \cV$.
By Corollary \ref{subgeocodes}, $B$ is a codeword of the code of the lines of $\pi$, so we can write it as a linear combination of some lines of  $\pi$, that is
$\sum_i \lambda_iv^{e_i}$, where $v^{e_i}$ are the characteristic vectors of the lines $e_i$ in  $\pi$.

Let  $B^*$ be the multiset  consisting of the lines $e_i$, with multiplicity $\lambda_i$, in the dual plane of $\pi$.
The weight of the codeword $B$ is  at most $q^3 + q^2 +q+1 $, hence in the dual plane this is the number of  lines intersecting $B^*$ in not $0\pmod{p}$ points.
Actually, as $B$ is a proper set, we know that each non $0\pmod{p}$ secant of $B^*$ must be a $1\pmod{p}$ secant.
Using Result \ref{bmost},
the number  of non $0\pmod{p}$ secants through any point is at most
$q+1$ or at least $q^2-q+1$. In the original plane $\pi$, this means that each line intersects $B$ in either at most $q+1$ or in at least $q^2-q+1$ points.
\end{proof}

\begin{proposition}
\label{charplanes1}
Assume that $\pi$ is a plane of $\PG(r,q^2)$, $q>4$, and $q = p^h$, $p$ prime, $h\geq 1$, such that  $|\pi \cap \cV|\leq q^3+2q^2$. Furthermore, suppose also
that there exists a line $\ell$ meeting $\pi \cap \cV$ in at least $q^2-q+1$ points, when $q^3 +1 \leq |\pi \cap \cV|$.
Then $\pi \cap \cV$ is a linear combination of  at most $q+1$  lines, each with weight $1$.
\end{proposition}

\begin{proof}
Let $B$ be the point set $\pi \cap \cV$.  By Corollary \ref{subgeocodes}, $B$ is the corresponding point set of a codeword $c$ of lines of $\pi$, that is  $c = \sum_i \lambda_iv^{e_i}$, where  lines of $\pi$ are denoted by $e_i$.
Let $C^*$ be the multiset in the dual plane containing the dual of each line $e_i$ with multiplicity $\lambda_i$. Clearly the number of lines intersecting $C^*$ in not $0\pmod{p}$ points
is $w(c )=|B|$. Note also, that every line that is not a $0\pmod{p}$ secant is a $1\pmod{p}$ secant, as $B$ is a proper point set, hence Property \ref{most} trivially holds (with $k=1$).

Our very first aim is to show that $c$ is a linear combination of at most $q+3$ different lines.
When $|B| < q^3+1$, then by Result \ref{codewords} it is a linear combination of at most $q$ different lines.


Next assume that $|B| \geq q^3+1$. From the assumption of the proposition, we know that  there exists a line $\ell$ meeting $\pi \cap \cV$ in at least $q^2-q+1$ points and from Lemma \ref{maxindex}, we also know that  each line intersects $B$ in either at most $q+1$ or in at least $q^2-q+1$ points.
Hence, if we add the line $\ell$ to $c$ with multiplicity $-1$, we reduce the weight by at least $q^2-q+1-q$ and at most by $q^2+1$.
If $w(c-v^\ell) < q^3+1$, then from the above we know that $c -v^\ell$ is a linear combination of  $\lceil {w(c-v^\ell)\over q^2+1}\rceil$ lines. Hence, $c$ is a linear combination of at most $q+1$ lines. If $w(c-v^\ell) \geq q^3+1$, then $w(c) \geq q^3+q^2-2q$ (see above) and so it follows that through any point of $B$, there passes at least one line intersecting $B$ in at least  $q^2-q+1$ points. This means that we easily find three lines $\ell_1$, $\ell_2$, and $\ell_3$ intersecting $B$ in at least  $q^2-q+1$ points. Since $w(c) \leq q^3+2q^2$, we get that $w(c - v^{\ell_1}- v^{\ell_2} - v^{\ell_3}) \leq q^3+2q^2 -3 \cdot( q^2-2q -2) < q^3+1$. Hence, similarly as before, we get that $c$ is a linear combination of at most $q+3$ lines.

Next we show that each line in the linear combination (that constructs $c$) has weight $1$. Take a line $\ell$ which is in the linear combination with coefficient $\lambda \not = 0$. Then there are at least $q^2+1-(q+2)$ positions, such that the corresponding point is in $\ell$ and the value at that position is $\lambda$. As $B$ is a proper set, this yields that $\lambda = 1$. By Remark \ref{1modpremark}, the number of lines with non-zero multiplicity in the linear combination of $c$ must be $1$ mod $\pmod{p}$, $p > 2$; hence it can be at most $q+1$.
\end{proof}

\begin{proposition}
\label{charplanes2}
Assume that $\pi$ is a plane of $\PG(r,q^2)$, $q>4$, and  $q = p^h$, $p$ prime, $h\geq 1$, such that  $|\pi \cap \cV|\leq q^3+2q^2$. Furthermore, suppose that
every line meets $ \pi\cap \cV$ in at most $q+1$ points.
Then $\pi \cap \cV$ is a classical unital.
\end{proposition}

\begin{proof}
Again let $B = \pi \cap \cV$ and first  assume that $|B| < q^3+1$. Proposition \ref{charplanes1} shows that $B$ is a linear combination of at most $q+1$ lines, each with weight $1$. But this yields that these lines intersect $B$ in at least $q^2+1-q$ points. So this case cannot occur.

Hence, $q^3+1\leq |B|\leq q^3+2q^2$.
We are going to prove that  there exists at least a tangent line to $B$ in $\pi$. Let $t_i$ be the number of lines meeting $B$ in $i$ points. Set $x=|B|$.  Then  double counting arguments give the following equations for  the integers $t_i$.

\begin{equation} \label{tg1}
\left\{\begin{array}{l}
    \sum_{i=1}^{q+1}t_i=q^4+q^2+1 \\
    \\
          \sum_{i=1}^{q+1}it_i=x(q^2+1)\\
     \\
     \sum_{i=1}^{q+1}i(i-1)t_i=x(x-1).
   \end{array}  \right.
  \end{equation}
Consider  $f(x)=\sum_{i=1}^{q+1}(i-2)(q+1-i)t_i$. From \eqref{tg1}, we get
\[f(x)=-x^2+x[(q^2+1)(q+2)+1]-2(q+1)(q^4+q^2+1).\]
Since $f(q^3/2)>0$, whereas $f(q^3+1)<0$ and $f(q^3+2q^2)<0$, it follows that
if $q^3+1\leq x\leq q^3+2q^2$, then $f(x)<0$ and thus  $t_1$ must be different from zero.
Therefore, $x=q^3+1$ and $\sum_{i=1}^{q+1}(i-1)(q+1-i)t_i=0$.

Since $(i-1)(q+1-i)> 0$, for $2\leq i\leq q$, we obtain $t_2=t_3=\cdots=t_q=0$, that is, $B$ is a set of $q^3+1$ points such that each line is a 1-secant or a $(q+1)$-secant of $B$.  Namely, $B$ is a unital and precisely a classical unital  since $B$ is a codeword of $\pi$ \cite{BBW}.
\end{proof}

The above two propositions and Lemma \ref{maxindex} imply the following corollary.

\begin{corollary}
\label{charunital}
Assume that $\pi$ is a plane of $\PG(r,q^2), q>4$ and  $q = p^h$, $p$ prime, $h\geq 1$, such that  $|\pi \cap \cV|\leq q^3+2q^2$.
Then $\pi \cap \cV$ is a linear combination of  at most $q+1$  lines, each with weight $1$, or it is a classical unital. \qed
\end{corollary}

\begin{corollary}
\label{charunitalbis}
Suppose that $\pi$ is a plane of $\PG(r,q^2), q>4$ and  $q = p^h$, $p$ prime, $h\geq 1$, containing exactly $q^3+1$ points of $\cV$.
Then $\pi \cap \cV$ is a classical unital.
\end{corollary}

\begin{proof}
Let $B$ be the point set $\pi \cap \cV$. We know that $B$ is the support of a codeword of lines of $\pi$.
By Proposition \ref{charplanes1}, if there is a line intersecting $B$ in at least  $q^2-q+1$ points, then $B$ is
a linear combination of at most $q+1$ lines, each with multiplicity $1$.
First of all note that a codeword that is a linear combination of $q+1$ lines has weight at least
$(q^2+1)(q+1) - 2{q+1\choose 2}$, that is exactly $q^3+1$. To achieve this, we need that the intersection points of any two lines from a linear combination are all different and the sum of the coefficients of any two lines is zero; which is clearly not the case (as all the coefficients are $1$).
From Remark \ref{1modpremark}, in this case $B$ would be a linear combination of at most $q+1-p$ lines and so its weight would be less than $q^3+1$, a contradiction.
Hence, there is no line intersecting $B$ in at least  $q^2-q+1$ points, so Proposition \ref{charplanes2} finishes the proof.
\end{proof}

\subsection{Case $r=3$}
In $\PG(3,q^2)$, each plane intersects $\cV$ in either $q^3+1$  or $q^3+q^2+1$ points since these are the intersection numbers of a quasi-Hermitian variety with a plane of $\PG(3,q^2)$.


\begin{lemma}\label{lem2}
Let $\pi$ be a plane in $\PG(3,q^2)$ such that $|\pi \cap \cV|=q^3+q^2+1$, then every line in $\pi$ meets $\pi \cap \cV$ in either $1$, $q+1$ or $q^2+1$ points.
\end{lemma}
\begin{proof}
Set $C=\pi \cap \cV$ and
let $m$ be  a line in $\pi$ such that $|m \cap C|=s$, with $s\neq 1$ and $s\neq q+1$. Thus, from Corollary \ref{charunitalbis},  every plane  through $m$ has to meet $\cV$ in $q^3+q^2+1$ points and thus
\[|\cV|=(q^2+1)(q^3+q^2+1-s)+s,\]
which gives $s=q^2+1$.
\end{proof}

\noindent {\bf Proof of Theorem \ref{mainth} (case $r = 3$): } From Corollary \ref{charunitalbis}  and Lemma \ref{lem2},  it follows that every line in $\PG(3,q^2)$ meets $\cV$ in either $1$, $q+1$, or $q^2+1$ points.
Now, suppose on the contrary that  there exists a singular point $P$ on $\cV$; this means that all lines through $P$ are either tangents or $(q^2+1)$-secants to $\cV$. Take a plane $\pi$ which does not contain $P$. Then $|\cV|=q^2|\pi \cap \cV|+1$ and since the two possible  sizes of the planar sections are $q^3+1$  or $q^3+q^2+1$, we get a contradiction.
Thus, every point in $\cV$ lies on at least one $(q+1)$-secant and,  from Theorem \ref{main1}, we obtain that  $\cV$ is a Hermitian surface. \qed

\subsection{Case $r \geq 4$ and $q=p$}

We first  prove the following result.
\begin{lemma}
If $\pi$ is a plane of $\PG(r,p^2)$, which is not contained in $\cV$, then either
\[|\pi \cap \cV|= p^2+1 \mbox{ or } \ |\pi \cap \cV| \geq p^3+1.\]

\end{lemma}
\begin{proof}
Let  $\pi$ be a plane of $\PG(r,p^2)$ and set $B= \pi \cap \cV$. By Remark \ref{1modpremark}, $B$ is a linear combination of $1$ mod $\pmod{p}$ not necessarily distinct lines.

If $|B|< p^3+1$, then by Result 2.5, $B$ is a linear combination of at most $p$ distinct lines. This and the previous observation yield that when
$|B|< p^3+1$, then it is the scalar multiple of one line; hence $|B|=p^2+1$.
\end{proof}

\begin{proposition}
\label{concurrent}
Let $\pi$ be a plane of $\PG(r,p^2)$, such that   $ |\pi \cap \cV| \leq p^3+p^2+p+1$.
Then $B=\pi \cap \cV$ is either a classical unital or a linear combination of $p+1$ concurrent lines or just one line, each with weight $1$.
\end{proposition}
\begin{proof}
 From Corollary \ref{charunital},  we have   that
$B$ is either a linear combination of at most $p+1$ lines or a classical unital.
In the first case, since $B$ intersects every line in $1\pmod{p}$ points and $B$ is a proper point set, the only possibilities are that $B$ is a linear combination of $p+1$ concurrent lines or just one line, each with weight $1$.
\end{proof}

\noindent {\bf Proof of Theorem \ref{mainth} (case $r\geq 4$, $q = p$): }
Consider a line $\ell$ of $\PG(r,p^2)$ which is not contained in $\cV$.
By Lemma \ref{lemma1}, there is a plane $\pi$ through $\ell$ such that $|\pi \cap \cV|\leq q^3+q^2+q+1$.   From Proposition \ref{concurrent}, we have that $\ell$ is either a unisecant or a $(p+1)$-secant of $\cV$ and we also have that $\cV$ has  no plane section of  size $(p+1)(p^2+1)$.   Finally, it is  easy to see  like in the  previous case $r=3$, that $\cV$ has no singular points, thus $\cV$ turns out to be a Hermitian variety of $\PG(r,p^2)$ (Theorem \ref{main2}).

\subsection{Case $r\geq 4$ and $q=p^2$}\label{alsin0}
Assume now that  $\cV$ is a quasi-Hermitian variety of $\PG(r,p^4)$, with $r\geq 4$.

Lemma \ref{maxindex} states that every line contains at most $p^2+1$ points of $\cV$ or at least $p^4-p^2+1$ points of $\cV$.
\begin{lemma}
\label{lem31} If $\ell$ is a line of $\PG(r,p^4)$, such that $|\ell \cap \cV|\geq p^4-p^2+1$, then $|\ell \cap \cV|\geq p^4-p+1$.
\end{lemma}
\begin{proof}
Set $|\ell \cap \cV|= p^4-x+1$, where $x\leq p^2$.  It suffices to prove that $x<p+2$. Let $\pi$ be a plane through $\ell$ and $B=\pi \cap \cV$. Choose  $\pi$  such that $|B|=|\pi \cap \cV|\leq  p^6+p^4+p^2+1$ (Lemma \ref{lemma1}).
Then,  by Proposition \ref{charplanes1}, $B$ is a linear combination of at most $p^2+1$ lines, each with weight 1.  Let $c$ be the codeword corresponding to $B$. We observe that  $\ell$ must be one of the lines of $c$,  otherwise $|B\cap \ell|\leq p^2+1$,  which is impossible.
Thus if $P$ is a point in $\ell \setminus B$, then through $P$ there pass at least  $p-1$ other lines of $c$.  If $x\geq p+2$, then the number of lines necessary to define the codeword $c$  would be at least $(p+2)(p-1)+1$, a contradiction.
\end{proof}

\begin{lemma}\label{lem32}
For each plane $\pi$ of $\PG(r,p^4)$, either $|\pi \cap  \cV|\leq p^6+2p^4-p^2-p+1$ or $ |\pi \cap  \cV|\geq p^8-p^5+p^4-p+1$.
\end{lemma}
\begin{proof}
Let $B=\pi \cap  \cV$,  $x=|B|$, and let $t_i$ be the number of lines in $\pi$ meeting $B$ in $i$ points. Then, in this case, Equations \eqref{tg1} read

\begin{equation} \label{tg2}
\left\{\begin{array}{l}
    \sum_{i=1}^{p^4+1}t_i=p^8+p^4+1 \\
    \\
          \sum_{i=1}^{p^4+1}it_i=x(p^4+1)\\
     \\
     \sum_{i=1}^{p^4+1}i(i-1)t_i=x(x-1).
   \end{array}  \right.
  \end{equation}

Set   $f(x)=\sum_{i=1}^{p^4+1}(p^2+1-i)(i-(p^4-p+1))t_i$. From \eqref{tg2} we obtain
\[f(x)=-x^2+[(p^4+1)(p^4+p^2-p+1)+1]x-(p^8+p^4+1)(p^2+1)(p^4-p+1). \]
 Because of Lemma \ref{lem31}, we get  $f(x)\leq 0$, while $f(p^6+2p^4-p^2+1)>0$, $f(p^8-p^5+p^4-p)>0$. This finishes the proof of the  lemma.
\end{proof}
\begin{lemma}\label{lem33}
If $\pi$ is a plane of $\PG(r,p^4)$, such that $ |\pi \cap  \cV|\geq p^8-p^5+p^4-p+1$, then  either $\pi$ is entirely contained in $\cV$ or
$ \pi \cap  \cV$ consists of $p^8-p^5+p^4+1$ points and it only contains $i$-secants, with $i\in \{1, p^4-p+1,p^4+1\}$.
\end{lemma}
\begin{proof}
Set $S=\pi \setminus \cV$.
Suppose that there exists some point $P \in S$. We have the following two possibilities: either each line of the pencil with center at  $P$ is a
$(p^4-p+1)$-secant or only one line through $P$ is an $i$-secant, with  $1\leq i\leq p^2+1$, whereas the other $p^4$ lines through $P$ are $(p^4-p+1)$-secants. In the former case, when there are no $i$-secants, $1\leq i\leq p^2+1$, each line  $\ell$  in $\pi$ either is disjoint from $S$ or it meets $S$ in $p$ points since  $\ell$ is a $(p^4-p+1)$-secant. This implies that $S$ is a maximal arc and this is impossible for $p\neq 2$ \cite{BB,BBM}.

In the latter case, we observe that the size of $\pi\cap \cV$ must be $p^8-p^5+p^4+i$, where $1\leq i\leq p^2+1$. Next, we denote by $t_s$ the number of $s$-secants in $\pi$, where $s\in \{i,p^4-p+1,p^4+1 \}$. We have that

\begin{equation} \label{tg3}
\left\{\begin{array}{l}
    \sum_{s}t_s=p^8+p^4+1 \\
    \\
          \sum_{s}st_s=(p^4+1)(p^8-p^5+p^4+i)\\
     \\
     \sum_{s}s(s-1)t_s=(p^8-p^5+p^4+i)(p^8-p^5+p^4+i-1).
   \end{array}  \right.
  \end{equation}
  From \eqref{tg3} we get
  \begin{equation} \label{tg3.1}
  t_i=\frac{p(p^4-p-i+1)(p^5-i+1)}{p(p^4-p-i+1)(p^4-i+1)}=\frac{p^5-i+1}{p^4-i+1}
  \end{equation}
and  we can see that
the only possibility for $t_i$ to be  an integer is $ip-p-i+1=0$, that is $i=1$.
For $i=1$, we get $|B|=p^8-p^5+p^4+1$.
\end{proof}

\begin{lemma}\label{lem4}
If $\pi$ is a plane of $\PG(r,p^4)$,  not contained in $\cV$ and  which does not  contain any  $(p^4-p+1)$-secant,  then  $\pi \cap  \cV$ is either a classical  unital or the union of $i$ concurrent lines, with $1\leq i\leq p^2+1$.
\end{lemma}
\begin{proof}
Because of Lemmas \ref{lem31}, \ref{lem32} and \ref{lem33}, the plane $\pi$  meets $\cV$ in at most $p^6+2p^4-p^2-p+1$ points. Furthermore,  each line of $\pi$ which is not contained in $\cV$ is an $i$-secant, with $1\leq i\leq p^2+1$ (Lemma \ref{lem31} and the sentence preceding Lemma \ref{lem31}). Set $B=\pi \cap \cV$. If in $\pi$ there are no $(p^4+1)$-secants to $B$, then $|B|\leq p^6+p^2+1$ and by Proposition \ref{charplanes2} it follows that  $B$ is a classical unital.

 If there is a $(p^4+1)$-secant to $B$ in $\pi$, then arguing as in the proof of Proposition \ref{charplanes1}, we  get that $B$ is  still a linear combination of $m$ lines, with  $m\leq p^2+1$. Each of these $m$ lines is a  $(p^4+1)$-secant to $\cV$. In fact if one of these lines, say $v$, was an $s$-secant, with $1\leq s\leq p^2+1$, then through each point $P\in v \setminus B$, there would pass at least $p$ lines of the codeword corresponding to $B$ and hence $B$ would be a linear combination of at least $(p^4+1-s)(p-1)+1>p^2+1$ lines, which is impossible.

We are going to prove that these $m$ lines, say $\ell_1,\ldots, \ell_m$,  are concurrent. Assume on the contrary that they are not. We can assume that through a point $P\in \ell_n$, there pass at least $p+1$ lines of our codeword but there is a line $\ell_j$ which does not pass through $P$. Thus through at least $p+1$ points on $\ell_j$, there are at least $p+1$ lines of our codeword and thus  we find at least $(p+1)p+1>m$  lines of $B$, a contradiction.
\end{proof}

\begin{lemma}\label{lem5}
A plane $\pi$ of $\PG(r,p^4)$  meeting $\cV$ in at most $p^6+2p^4-p^2-p+1$ points and  containing a $(p^4-p+1)$-secant  to $\cV$ has at most $(p^2+1)( p^4-p+1)$ points.
\end{lemma}
\begin{proof} Let $\ell$ be a line of $\pi$  which is  a $(p^4-p+1)$-secant to $\cV$. In this case, $\pi \cap \cV$ is a linear combination of at most $p^2+1$ lines, each with weight 1 (Proposition \ref{charplanes1}). In particular $\ell$ is a line of the codeword and hence through each of the missing points of $\ell$ there are at least $p$ lines of the codeword corresponding to $B$. On these $p$ lines we can see at most $p^4-p+1$ points of $\cV$.

So let $\ell_1, \ell_2, \ldots, \ell_p$ be $p$ lines of the codeword through a point of $\ell$. Each of these lines contains at most $p^4-p+1$ points of $\cV$. Thus these $p$ lines contain together at most $ p(p^4-p+1)$ points of $\cV$.  Now take any other line of the codeword, say $e$. If $e$ goes through the common point of the lines $\ell_i$, then there is already one point missing from $e$, so adding $e$ to our set, we can add at most $p^4-p+1$ points. If $e$ does not go through the common point, then it intersects $\ell_i$ in $p$  different points. These points  either do not belong to the set $\pi\cap \cV$ or they belong to the set $\pi\cap \cV$, but we have already counted them when we counted the points of $\ell_i$, so again $e$ can add at most  $p^4-p+1$ points to the set $\pi\cap \cV$.
 Thus adding the lines of the codewords one by one to $\ell_i$ and counting the number of points, each time we add only at most $p^4-p+1$ points to the set $\pi\cap \cV$.
Hence, the plane $\pi$ contains at most $(p^2+1)(p^4-p+1)$ points of $\cV$.
\end{proof}

\begin{lemma}\label{lem51}
Let $\pi$  be a plane of $\PG(r,p^4)$, containing  an $i$-secant, $1< i < p^2 + 1$,  to $\cV$. Then
$\pi \cap \cV$ is either the union of $i$ concurrent lines or it is a linear combination of $p^2+1$ lines (each with weight 1) so that they form a subgeometry of order $p$, minus $p$ concurrent lines.
\end{lemma}
\begin{proof}
 By Lemmas \ref{lem32}, \ref{lem33} and \ref{lem4},
 $\pi$  meets $\cV$ in  at most $p^6 +2p^4 -p^2 -p+1$ points and must contain a $(p^4-p+1)$-secant  to $\cV$ or $\pi \cap \cV$ is the union of $i$ concurrent lines. Hence from now on, we assume that $\pi$ contains a  $(p^4 -p+1)$-secant. By Result \ref{codewords}, such a plane is a linear combination of at most $p^2+1$ lines.
As before each line from the linear combination has weight $1$.
Note that the above two statements imply that a line of the linear combination will be either a $(p^4 +1)$-secant or a $(p^4 -p+1)$-secant.

As we have  at most $p^2+ 1$ lines, the line of a $(p^4 -p+1)$-secant must be one of the lines from the linear combination. This also means that through each of the $p$ missing points of this line, there must pass at least $p-1$ other lines from the linear combination. Hence, we already get $(p-1)p +1$ lines.

In the case in which  the linear combination contains exactly $p^2-p+1$ lines, then from each of these lines there are exactly $p$ points missing and through each missing point there are exactly $p$ lines from the linear combination. Hence, the missing points and these lines form a projective plane of order $p-1$, a contradiction as $p>3$.

Therefore, as the number of the lines of the linear combination must be $1$ $\pmod{p}$ and at most $p^2+1$, we can assume that  the linear combination contains  $p^2 +1$ lines.
We are going to prove that through each point of the plane there pass either $0$, $1$, $p$ or $p+1$ lines from the linear combination. From earlier arguments, we know that the number of lines through one point $P$ is $0$ or $1$ $\pmod{p}$. Assume to the contrary that through $P$ there pass at least $p+2$ of such lines. These $p^2+1$ lines forming the linear combination are not concurrent, so there is a line $\ell$ not through $P$. Through each of the intersection points of $\ell$ and a line through $P$, there pass at least $p-1$ more other lines of the linear combination, so in total we get at least $(p-1)(p+2)+1$ lines forming the linear combination, a contradiction.

Since there are $p^2 +1$ lines forming the linear combination and through each point of the plane there pass either $0$, $1$, $p$ or $p+1$ of these lines, we obtain that on a $(p^4-p+1)$-secant there is exactly one point, say $P$,  through which there pass exactly $p+1$ lines from the linear combination and $p$ points, not in the quasi Hermitian variety, through each of which there pass exactly $p$ lines.

If all the $p^2+1$ lines forming the linear combination,  were $(p^4-p+1)$-secants then the number of points through which there pass exactly $p$ lines would be $(p^2+1)p / p$. On the other hand, through $P$ there pass $p+1$  $(p^4-p+1)$-secants, hence we already get $(p+1)p$ such points, a contradiction.
 Thus, there exists a line $m$ of the linear combination that is a $(p^4+1)$-secant. From the above arguments, on this line there are exactly $p$ points through each of which there pass exactly $p+1$ lines, whereas through the rest of the points of the line $m$ there pass no other lines of the linear combination.

 Assume that there is a line $m' \neq m$ of the linear combination that is also a $(p^4+1)$-secant. Then there is a point $Q$ on $m'$ but not on $m$ through which there pass $p+1$ lines. This would mean that there are at least $p+1$ points on $m$, through which there pass more than $2$ lines of the linear combination, a contradiction.

 Hence, there is exactly one line $m$ of the linear combination that is a $(p^4+1)$-secant and all the other lines of the linear combination are $(p^4-p+1)$-secants.
It is easy to check  that the points through which there are more than $2$ lines plus the $(p^4-p+1)$-secants form a dual affine plane.
Hence our lemma follows.
\end{proof}

\begin{lemma}\label{lem35}
There are no $i$-secants to $\cV$, with $1<i<p^2+1$.
\end{lemma}

\begin{proof}
By Lemma \ref{lem51},
if a plane $\pi$ contains an $i$-secant, $1 < i < p^2 +1$, then $\pi \cap \cV$ is a linear combination of either  $i$ concurrent lines or lines of an embedded subplane of order $p$ minus $p$ concurrent lines. In the latter case, if $i>1$ then  an $i$-secant is at least a $(p^2-p+1)$-secant.
Hence, if there  is  an $i$-secant with $1<i<p^2-p+1$, say $\ell$, we get that  for each plane $\alpha$  through $\ell$,  $\alpha \cap \cV$ is a linear combination of $i$ concurrent lines.
 Therefore
  \begin{equation}\label{eq17}
|\cV|= m(ip^4+1-i)+i,
\end{equation}
where $m=p^{4(r-2)}+p^{4(r-3)}+ \ldots +p^4+1$ is the number of planes in $\PG(r,p^4)$ through $\ell$.

Setting $r=2\sigma+\epsilon$, where $\epsilon=0 $ or $\epsilon=1$ according to $r$ is even or odd, we can write
\[|\cV|=1+p^4+\ldots+p^{4(r-\sigma-1)}+(p^{4(r-\sigma-\epsilon)}+p^{4((r-\sigma-\epsilon+1)}+\ldots+p^{4(r-1)})p^2\]

Hence, \eqref{eq17} becomes
   \begin{equation}
   \begin{array}{l}\label{eq18}
1+p^4+\ldots+p^{4(r-\sigma-1)}+(p^{4(r-\sigma-\epsilon)}+p^{4((r-\sigma-\epsilon+1)}+\ldots+p^{4(r-1)})p^2\\
-(p^{4(r-2)}+p^{4(r-3)}+ \ldots +p^4+1)=ip^{4(r-1)}
\end{array}
\end{equation}
Since $\sigma \geq 2$, we see that  $p^{4(r-1)}$ does not divide the left hand side of (\ref{eq18}),
  a contradiction.

Thus, there can only  be  $1$-, $(p^2-p+1)$-, $(p^2+1)$-, $(p^4-p+1)$- or $(p^4+1)$-secants to $\cV$.
Now, suppose  that $\ell$ is a $(p^2-p+1)$-secant to $\cV$. Again by Lemma  \ref{lem51},
 each plane through $\ell$  either  has $x=(p^2-p+1)p^4+1$ or  $y=p^2(p^4-p)+p^4+1$ points of $\cV$.
 Next, denote by $t_j$ the number of $j$-secant planes through $\ell$  to  $\cV$.
 We get
 \begin{equation} \label{2c}
\left\{\begin{array}{l}
    t_{x}+t_{y}=m \\
    t_{x}(x-p^2+p-1)+t_{y}(y-p^2+p-1)+p^2-p+1=|\cV|
   \end{array}  \right.
  \end{equation}
Recover the value of $t_y$ from the first
equation and substitute it in the second.
 We obtain
 \[(m-t_y)(p^6-p^5+p^4-p^2+p)+t_y(p^6+p^4-p^3-p^2+p)+p^2-p+1=|\cV|\]
 that is,
\[p^3(p^2-1)t_y=|\cV|-m(p^6-p^5+p^4-p^2+p)-p^2+p-1.\]
It is easy to check that $|\cV|-m(p^6-p^5+p^4-p^2+p)-p^2+p-1$ is not divisible by $p+1$ and hence, $t_y$  turns out not to be  an integer, which is impossible.
 \end{proof}

\begin{lemma} \label{3.20}No plane meeting $\cV$ in at most $p^6+2p^4-p^2-p+1$ points contains a $(p^4-p+1)$-secant.
\end{lemma}

\begin{proof} Let $\pi$ be a plane of $\PG(r,p^4)$ such that $|\pi \cap \cV|\leq p^6+2p^4-p^2-p+1$. It can contain only $1$-, $(p^2+1)$-, $(p^4-p+1)$-, $(p^4+1)$-secants (Lemma \ref{lem31} and Lemma \ref{lem35}).
If $\pi\cap \cV$ contains a $(p^4-p+1)$-secant, we know from Proposition \ref{charplanes1} that it is
a linear combination of at most $p^2+1$ lines, each with weight 1. Suppose that  $e$  is a $(p^4-p+1)$-secant to $\pi \cap \cV$.
Let $P$ and $Q$ be two missing points of $e$. We know that there must be at least $p-1$ other lines of the codeword through $P$ and $Q$. Let $f$ and $g$ be two such lines through $Q$. We can find a line, say $m$, of the plane through $P$, that intersects $f$ and $g$ in a point of $\cV$ and that is not a line of the codeword. Then $|m \cap \cV| \geq 1+p$ since $|m \cap \cV| \equiv 1 \pmod{p}$.
 Thus $m$ contains at least two points of $\cV$, but in $P$ it meets at least $p$ lines of the codeword. Hence,  $p+1\leq |m\cap \cV|\leq p^2-p+1$, and this contradicts Lemma \ref{lem35}.
\end{proof}

\begin{lemma}\label{lem34}
There are no $(p^4-p+1)$-secants to $\cV$.
\end{lemma}
\begin{proof} If there was a $(p^4-p+1)$-secant to $\cV$, say $\ell$, then, by Lemma \ref{3.20}, all the planes through $\ell$  would contain at least $p^8-p^5+p^4+1$ points of $\cV$, and thus
\begin{equation}\label{alsin}|\cV|\geq (p^{4(r-2)}+p^{4(r-3)}+ \cdots +p^4+1)(p^8-p^5+p)+p^4-p+1,
\end{equation} a contradiction.
\end{proof}

\noindent {\bf Proof of Theorem \ref{mainth} (case $r \geq 4$ and $q = p^2$): }
Consider a line $\ell$ which is not contained in $\cV$. From the preceding lemmas we have that $\ell$ is either a 1-secant or a $(p^2+1)$-secant of $\cV$. Furthermore, $\cV$ has  no plane section of  size $(p^2+1)(p^4+1)$.   Finally, as in the case $r=3$,  it is  easy to see that $\cV$ has no singular points, thus, by Theorem \ref{main2}, $\cV$ turns out to be a Hermitian variety of $\PG(r,p^4)$.

\section{Singular quasi-Hermitian varieties}\label{sqHv}
In this section, we consider sets having the same behavior with respect to hyperplanes as singular Hermitian varieties.
\begin{definition}
A {\em $d$-singular quasi-Hermitian variety} is a subset of points of $\PG(r,q^2)$ having the same number of points and the same intersection sizes with hyperplanes as a singular Hermitian variety with a singular space of dimension $d$.
\end{definition}

We prove the following result.
\begin{theorem}
Let $\mathcal{S}$ be a $d$-singular quasi-Hermitian variety in $\PG(r,q^2)$.
Suppose that either
\begin{itemize}
\item  $r=3$, $d=0$, $q=p^{h}> 4$, $h\geq 1$, or
\item $r\geq 4$, $d\leq r-3$, $q=p>4$, or
\item  $r\geq 4$, $d\leq r-3$, $q=p^2$, $p>3$.
\end{itemize}
Then $\mathcal{S}$ is a singular Hermitian variety with a singular space of dimension $d$ if and only if its incidence vector  is in the $\mathbb{F}_p$-code spanned by the hyperplanes of $\PG(r,q^2)$.
\end{theorem}
\begin{proof}
Let $\mathcal{S}$ be  a singular Hermitian variety of $\PG(r,q^2)$. The  characteristic vector $v^{\cS}$ of $\cS$  is in $ \mathcal C_p$ since
   \cite[Theorem 1]{JK} also holds for singular Hermitian varieties.
 Now assume that $\cS$ is a $d$-singular quasi-Hermitian variety. As in the non-singular case, by Lemma \ref{Lemma:rette}, each line of $\PG(r,q^2)$ intersects $\mathcal{S}$ in $1\pmod{p}$ points.

\subsection{Case $r=3$}

Suppose that $r=3$ and therefore $d=0$. Let $\pi$   be a plane  of $\PG(3,q^2)$.
In this case  $\pi$   meets $\mathcal{S}$ in either $q^2+1$, or $q^3+1$ or  $q^3+q^2+1$ points. Therefore Lemma \ref{lemma1},
Lemma   \ref{maxindex},
Proposition \ref{charplanes1},  and Corollary \ref{charunital}, are still valid in the singular case for $r=3$.

Thus, if $|\pi \cap \mathcal{S}|=q^2+1$, then Proposition \ref{charplanes1} implies that $\pi \cap \mathcal{S}$ is a line of $\pi$,  whereas if $|\pi \cap \mathcal{S}|=q^3+1$, then Corollary \ref{charunital} gives that $\pi \cap \mathcal{S}$  is a classical unital of $\pi$.
Now suppose that $|\pi \cap \mathcal{S}|=q^3+q^2+1$. Let $\ell$ be a line of $\pi$ such that $|\ell \cap \mathcal{S}|=s$ with $s\neq 1, q+1,q^2+1$.

 Each plane through $\ell$ must meet $\mathcal{S}$ in
$ q^3+q^2+1$ points and this gives
\[(q^3+q^2+1-s)(q^2+1)+s=q^5+q^2+1\]
that is, $ s=q^2+q+1$ which is impossible.

Thus each line of $\PG(3,q^2)$ intersects $\mathcal{S}$
in  either $1$ or, $q+1$ or, $q^2+1$ points and hence $\mathcal{S}$ is a $k_{q+1,3,q^2}$.
 Also, $\mathcal{S}$ cannot be non-singular by assumption, hence  Theorem \ref{th:singularcase1} applies and $\cS$ turns out to be a cone $\Pi_0\mathcal{S}^{\prime}$ with $\mathcal{S}^{\prime}$ of type {I}, {II}, {III} or {IV } as   the possible intersection sizes with planes are $q^2+1,q^3+1,q^3+q^2+1$.

 Possibilities {II}, {III}, and {IV} must be excluded, since their sizes cannot be possible. This implies that $ \mathcal{S}=\Pi_0\mathcal{H}$, where $\mathcal{H}$ is a non-singular Hermitian curve.

\subsection{Case $r\geq 4$}
 Let  $\ell$ be a line of $\PG(r,q^2)$  containing $x<q^2+1$ points of $\mathcal{S}$. We are going to prove that there exists at least one plane through $\ell$  containing less than $q^3+q^2+q+1$ points of $\mathcal{S}$.
If we suppose that all the planes through $\ell$ contain at least $q^3+q^2+q+1$ points of $\mathcal{S}$, then
$$
q^{2(d+1)}\frac{(q^{r-d}+(-1)^{r-d-1})(q^{r-d-1}-(-1)^{r-d-1})}{q^2-1} + q^{2d}+q^{2(d-1)}+\cdots+q^2+1 \geq$$
$$ m (q^3+q^2+q+1-x)+x,
$$
where $m=q^{2(r-2)}+q^{2(r-3)}+\cdots+q^2+1$  is the number of planes through $\ell$ in $\PG(r,q^2)$. We obtain  $x>q^2+1$, a contradiction.

Therefore, there exists at least one plane through $\ell$ having less than $q^3+q^2+q+1$ points of $\mathcal{S}$ and hence Lemma \ref{lemma1}, Lemma   \ref{maxindex},
Proposition \ref{charplanes1},  and Corollary \ref{charunital}, are still valid in this singular case for any $q>4$.

Next, we are going to prove that $\cS$ is a $k_{q+1,r,q^2}$, with $q=p^{h}>4$, $h=1,2$.
\noindent{\bf case $q = p$: }
Let $\ell$ be a line of $\PG(r,p^2)$. As we have seen there is a plane $\pi$ through $\ell$ such that $|\pi \cap \cV|\leq p^3+p^2+p+1$.   Proposition \ref{concurrent} is still valid in this case and thus we have that $\ell$ is either a unisecant or a $(p+1)$-secant of $\cS$. Furthermore, we also have that $\cS$ has  no plane section of  size $(p+1)(p^2+1)$ and hence $\cS$ is a regular $k_{p+1,r,p^2}$.

\noindent{\bf case $q = p^2$: }
We first observe that \eqref{eq17} and \eqref{alsin} hold true  in the case in which $\cV$ is assumed to be a singular quasi-Hermitian variety. This implies that all lemmas stated in the subparagraph \ref{alsin0}  are valid in our case. Thus,
 we obtain that $\mathcal{S}$ is a $k_{p^2+1,r,p^4}$ and
  it is straightforward to check that $\mathcal{S}$ is also  regular.

 Finally,  in both cases $q=p$ or $q=p^2$,  we have that $\cS$ is  a singular $k_{q+1,r,q^2}$ because  if $\mathcal{S}$ were a non-singular $k_{q+1,r,q^2}$ then, from Theorem \ref{main2}, $\mathcal{S}$ would be  a non-singular Hermitian variety and this is not possible by our assumptions.

 Therefore, by Theorem \ref{th:singularcase}, the only possibility is that $\mathcal{S}$ is a cone $\Pi_d \mathcal{S}^\prime$, with $\mathcal{S}^{\prime}$ a non-singular $k_{q+1,r-d-1,q^2}$. By Lemma \ref{CDS}, $\mathcal{S}^{\prime}$ belongs to the code of points and hyperplanes of $\PG(r-d-1,q^2)$. Since $r-d-1\geq 2$, then, by \cite{BBW} and Theorem \ref{mainth}, $\mathcal{S}^{\prime}$ is a non-singular Hermitian variety and, therefore, $\mathcal{S}$ is a singular Hermitian variety with a vertex of dimension $d$.
\end{proof}

{\bf Acknowledgement.} The research of the second author  was partially supported by Ministry for Education, University and Research of Italy (MIUR) (Project PRIN 2012 ``Geometrie di Galois e strutture di incidenza'' - Prot. N. 2012XZE22K\_005) and by the Italian National Group for Algebraic and Geometric Structures and their Applications (GNSAGA - INdAM). The third and fourth author acknowledge the financial support of the Fund for Scientific Research - Flanders (FWO) and the Hungarian Academy of Sciences (HAS) project: Substructures of projective spaces (VS.073.16N).

\end{document}